\documentclass[11pt,a4paper]{article}
\usepackage{hyperref}
\usepackage[english]{babel}
\usepackage[utf8]{inputenc}
\usepackage[T1]{fontenc}
\usepackage{amsmath,amsthm,amssymb,amsfonts}
\numberwithin{equation}{section}
\usepackage{enumerate}
\usepackage{faktor}
\usepackage{dsfont}
\usepackage{scalerel,stackengine}
\usepackage{textcomp}
\usepackage{graphicx}
\usepackage{extarrows}
\usepackage{epigraph}
\usepackage{graphicx}
\usepackage{subcaption}
\usepackage{sidecap}
\usepackage{indentfirst}
\usepackage{tikz}
\usepackage{tikz-cd}
\usetikzlibrary{shapes.misc,arrows,matrix,patterns,decorations.markings,positioning}
\tikzset{cross/.style={cross out, draw=black, minimum size=2*(#1-\pgflinewidth), inner sep=0pt, outer sep=0pt},
cross/.default={4.5pt}}
\usepackage{pgfplots}
\usepackage{caption}
\usepackage{float}
\usepackage{color}
\usepackage{epstopdf}
\usepackage{epsfig}
\usepackage{stmaryrd}
\usepackage{color}
\usepackage{geometry}
\usepackage{multicol}
\usepackage{verbatim}

\DeclareMathOperator{\lk}{\ell k}

\renewcommand{\geq}{\geqslant}
\renewcommand{\leq}{\leqslant} 
\renewcommand{\epsilon}{\varepsilon}

\newcommand{\Z}{\mathbb{Z}}

\newcommand{\C}{\mathbb C}
\newcommand{\PP}{\mathbb P}

\DeclareFontFamily{U}{mathx}{\hyphenchar\font45}
\DeclareFontShape{U}{mathx}{m}{n}{
      <5> <6> <7> <8> <9> <10>
      <10.95> <12> <14.4> <17.28> <20.74> <24.88>
      mathx10
      }{}
\DeclareSymbolFont{mathx}{U}{mathx}{m}{n}
\DeclareFontSubstitution{U}{mathx}{m}{n}
\DeclareMathAccent{\widecheck}{0}{mathx}{"71}
\DeclareMathAccent{\wideparen}{0}{mathx}{"75}

\newtheorem{thm}{Theorem}[section]
\newtheorem{lemma}[thm]{Lemma}
\newtheorem{prop}[thm]{Proposition}
\newtheorem{defin}[thm]{Definition}
\newtheorem{cor}[thm]{Corollary}

\usepackage{xpatch}
\makeatletter
\xpatchcmd{\@thm}{\thm@headpunct{.}}{\thm@headpunct{}}{}{}
\makeatother

\theoremstyle{definition}
\newtheorem{remark}[thm]{Remark}

\pgfplotsset{compat=1.13}
\begin{document}
\title{Traces of links and simply connected 4-manifolds}
\author{\scshape{Alberto Cavallo and Andr\'as I. Stipsicz}\\ \\
 \footnotesize{Alfr\'ed R\'enyi Institute of Mathematics,}\\
 \footnotesize{Budapest 1053, Hungary}\\ \\ \small{cavallo.alberto@renyi.hu}\\ \small{stipsicz.andras@renyi.hu}}
\date{}

\maketitle

\begin{abstract}
 We study the set $\widehat{\mathcal S}_M$ of framed smoothly slice
 links which lie on the boundary of the complement of a 1-handlebody
 in a closed, simply connected, smooth 4-manifold $M$. We show that $\widehat{\mathcal S}_M$
 is well-defined and describe how it relates to exotic phenomena in
 dimension four. In particular, in the case when $X$ is smooth, with a
 handle decompositions with no 1-handles and homeomorphic to but not
 smoothly embeddable in $D^4$, our results tell us that $X$ is exotic
 if and only if there is a link $L\hookrightarrow S^3$ which is
 smoothly slice in $X$, but not in $D^4$.
 
 Furthermore, we extend the notion of high genus 2-handle attachment,
 introduced by Hayden and Piccirillo, to prove that exotic 4-disks 
 that are smoothly embeddable in $D^4$, and therefore possible
 counterexamples to the smooth 4-dimensional Sch\"onflies conjecture,
 cannot be distinguished from $D^4$ only by comparing the slice genus
 functions of links.
\end{abstract}

\section{Introduction}
The smooth 4-dimensional Poincar\'e conjecture 4SPC (asserting that
any smooth 4-manifold homeomorphic to the 4-sphere $S^4$ is
diffeomorphic to it) is one of the most important and well-studied
problems in topology. The difficulty of the problem stems from the fact that
currently we do not have smooth invariants directly applicable to 4-manifolds
homotopy equivalent to $S^4$.

It is known that the existence of an exotic 4-sphere (a counterexample
to 4SPC) is equivalent to the existence of an exotic 4-disk, i.e. a
smooth 4-manifold with boundary $S^3$ which is homeomorphic to the
4-disk $D^4$, but not diffeomorphic to it.  For such 4-manifolds we
have the hope to be able to distinguish them based on the feature that
knots (and links) might have different slice genera in them, see \cite{4,MMP,MP}.  This
observation leads us to the following definition.  (We fix the
convention that every manifold is oriented and diffeomorphisms, denoted by $\cong$, between
manifolds always preserve the given orientations.)
\begin{defin}
  Suppose that $X$ is a compact, simply connected, smooth $4$-manifold
  with $\partial X=S^3$. Let $\mathcal S_X$ denote the set of
  smoothly slice links in $X$, i.e. the set of those links in $S^3$,
  which bound smoothly, disjointly and properly embedded $2$-disks
  in $X$.
\end{defin}

In trying to understand potential exotic $D^4$'s, we can divide them
into two groups (as it has been done for exotic ${\mathbb {R}}^4$'s).
\begin{defin}
  Suppose that $X$ is a possibly exotic $4$-disk. $X$ is {\bf \emph{small}},
  if there is a smooth embedding $f:X\hookrightarrow D^4$; otherwise $X$ is
  {\bf \emph{large}}.
\end{defin}

Our first observation shows that small exotic 4-disks cannot
be detected using slice links.
\begin{prop}
 \label{prop:smallexotic}
  Suppose that $X$ is a small exotic $4$-disk. Then $\mathcal S_X$ and
  $\mathcal S_{D^4}$ are equal.
\end{prop}
\begin{remark}
  The existence of a small exotic 4-disk is equivalent to the
  failure of the well-studied smooth 4-dimensional Sch\"onflies
  conjecture, asserting that a smoothly embedded 3-sphere in
  $S^4$ bounds a smoothly embedded $D^4$.
  \end{remark}

Based on the result of Proposition~\ref{prop:smallexotic} one can hope
that large exotic $D^4$'s can be detected by their set of slice
links. Indeed, this result applies for a special class of exotic
disks. Recall that a 4-manifold is \emph{geometrically simply
  connected} if it admits a handle decomposition without 1-handles.
Obviously, geometrically simply connected manifolds are simply
connected, but the converse does not hold. For example, a compact,
contractible 4-manifold $X$ with boundary $\partial X\neq S^3$
(according to a result of Casson, see \cite{Kirby}) is never
geometrically simply connected.  The question whether simple
connectivity implies geometrical simple connectivity for closed
4-manifolds, is wide open.

Suppose that $L$ is an $n$-component link. Let $S^3_{(0, \ldots
  ,0)}(L)$ denote the 3-manifold we get by performing 0-surgery on
each component of $L$. Let $X(L)$ denote the corresponding surgery
\emph{trace}, that is, the 4-manifold given by attaching 0-framed 2-handles
along the components of $L$.

There is a simple way to construct an exotic 4-sphere once two knots
$K_1$ and $ K_2$ with diffeomorphic 0-surgery are given, where one of
the knots (say $K_1$) is smoothly slice, while $K_2$ is not. Indeed,
glue the complement of the slice disk of $K_1$ to the 0-trace
$X(K_2)$; the application of the trace embedding lemma
(cf. Lemma~\ref{lemma:trace}) shows that the result is exotic. This
construction, which is used in \cite{MP}, admits a natural generalization to links, where
Lemma~\ref{lemma:trace} provides the simple extension of the usual
trace embedding lemma to links.  As the next results show, this
construction is sufficient to produce all geometrically simply
connected, large, exotic 4-disks.

\begin{thm}
 \label{thm:GSCPC}
 Suppose that $L\hookrightarrow S^3$ is an $n$-component link such
 that $S^3_{(0,...,0)}(L)\cong\#^nS^1\times S^2=Y$ and consider the
 smooth $4$-manifold given by $S=X(L)\cup_Y\natural^nS^1\times
 D^3$. Then $S$ is a geometrically simply connected homotopy
 $4$-sphere.  Furthermore, every geometrically simply connected exotic
 $4$-sphere $S$ is constructed in this way and $S$ is exotic if and
 only if $X(L)$ is not diffeomorphic to $\natural^nS^2\times D^2$.
\end{thm}

As it was indicated above, for a geometrically simply connected large
exotic 4-disk $X$ the set $\mathcal S_X$ is sufficient to verify its exoticness.

\begin{thm}
 \label{thm:geomsimpconn}
  A geometrically simply connected $4$-disk $X$ is small if and only
  if $\mathcal S_X=\mathcal S_{D^4}$. Consequenly, a geometrically
  simply connected exotic $4$-disk $X$ is large if and only if there
  is a link $L\hookrightarrow S^3$ that is smoothly slice in $X$, but
  not in $D^4$.
  \end{thm}

The definition of the set $\mathcal S_X$ can be extended to contain
links with various framings, and even in other 3-manifolds (such as
$\#^k S^1\times S^2$) naturally associated to $X$, providing the set
of framed links $\widehat{\mathcal {S}}_X$. We will discuss these
extensions (and the precise description of $\widehat{\mathcal {S}}_X$)
in Section~\ref{sec:frames}, leading us to the following result.

\begin{thm}
 \label{thm:type}
 If the smooth $4$-dimensional Poincar\'e conjecture holds then the
 set $\widehat{\mathcal S}_X$ always determines the diffeomorphism
 type of a closed, simply connected, smooth $4$-manifold $X$.
\end{thm}

A further refinement of the set of slice links is to consider the
(slice) genus function on the set of links in the boundary of a
4-manifold. This naturally extends the usual genus function of a
closed 4-manifold to the setting of 4-manifolds with boundary,
and this concept can be interesting even in the case when the manifold
at hand has no second homology. This approach requires the extension
of the trace embedding lemma to be relevant in this context; in
particular, we will attach higher genus handles (as it has been
already considered in~\cite{HP}) along framed knots and links,
and will consider higher order traces (when the core surfaces of the
handles have potentially multiple boundary components).

Equipped with these tools, we can show that even the genus function cannot
detect small exotic 4-disks.
  \begin{thm}
 \label{thm:slice}
 If $X$ is a small exotic $4$-disk and $L$ is a link in $S^3$ then
 $g_4(L)=g_4^X(L)$ and $\chi_4(L)=\chi_4^X(L)$, where $g_4^X$ and
 $\chi _4^X$ denote the genus functions in $X$, with the convention
 that for $X=D^4$ we omit the superscript.
\end{thm}

The paper is composed as follows: in Section~\ref{sec:prelims}
we list a few preliminary results, and show that small exotic 4-disks
cannot be detected using sliceness of links.
In Section~\ref{sec:traces} we focus on geometrically simply connected
exotic 4-disks, and in Section~\ref{sec:frames} we extend the notion of
$\mathcal S_X$ to framed links and show that (assuming 4SPC) these sets
characterize closed, simply connected, smooth 4-manifolds. Finally, in Section~\ref{sec:higherorder}
we discuss the genus function, and show that even this invariant cannot
distinguish a small exotic 4-disk from $D^4$.

\bigskip

\noindent {\bf Acknowledgement}:
Both authors were supported by the \emph{\'Elvonal} Grant KKP 126683 provided
by NKFIH.

\section{Preliminaries}
\label{sec:prelims}
It follows from the Cerf-Palais lemma \cite{Cerf,Palais} that 
in a closed, connected, smooth (resp. topological) $n$-manifold $M$
 all smooth (resp. locally flat) embeddings $D^n\hookrightarrow M$ are
 isotopic.  This statement then easily implies
 that for a smooth embedding $\iota \colon D^4\hookrightarrow S^4$ we have
 $S^4\setminus\mathring{\iota(D^4)}\cong D^4$.
This result shows that for every pair of embedded $D^4$ in a given
homotopy 4-sphere there is a diffeomorphism that sends one into the
other. This observation immediately leads to a relation
between exotic 4-disks and homotopy 4-spheres; here, for an exotic 4-disk
$X$, we denote $\widehat X$ the exotic 4-sphere which is gotten by
gluing a 4-handle to $X$.
\begin{prop}
 \label{prop:bs}
 There is a one-to-one correspondence between exotic $4$-disks and homotopy $4$-spheres, in the sense that two exotic $4$-disks $X_1$ and $X_2$ are diffeomorphic if and only if the same is true for $\widehat{X_1}$ and $\widehat{X_2}$.
\end{prop}
\begin{proof}
 If $X_1\cong X_2$ then clearly $\widehat{X_1}\cong\widehat{X_2}$ because there is a unique way to glue a 4-handle to a manifold with $S^3$ as boundary. 
 The other implication follows from our observation above.
\end{proof}
In accordance with the identification given in Proposition
\ref{prop:bs}, we also say that a homotopy 4-sphere $\widehat X$ is
small or large when the corresponding $X$ is.  If $X\hookrightarrow
S^4$ is a small exotic 4-disk then the same is true for
$S^4\setminus\mathring X$. Hence, we can put the structure of an
abelian group on the set of small homotopy 4-spheres up to
diffeomorphism, where the group operation is given by the boundary sum
$\natural$. Very little is known about this group, except that it has
at most a countable number of elements, see \cite{GS}.

The following is a simple, yet useful
generalization of Proposition~\ref{prop:bs}. 
\begin{lemma}
 \label{lemma:bouquet}
 Suppose that $X_1$ and $X_2$ are two compact smooth $4$-manifolds
 with $\partial X_1\cong\partial X_2\cong\#^nS^1\times S^2$ for some
 $n$; and denote the manifold obtained by gluing $\natural^nS^1\times
 D^3$ to $X_i$ for $i=1,2$ by $M_i$. Then, if $M_i$ are simply
 connected and satisfy $M_1\cong M_2$, then $X_1\cong X_2$.
\end{lemma}
\begin{proof}
The complements of $X_1$ and of $X_2$ in $M=M_1(\cong M_2)$ are both
neighborhoods of bouquets of circles (of the same number, as this
number is equal to $n$). By an isotopy of $M$ we can arrange that the
two bouquets have the same 0-cell. As $M$ is simply connected, the
circles of the bouquets are homotopic to each other. In this
dimension, however, homotopy implies isotopy \cite{Milnor}, ultimately
providing an isotopy from $X_1$ to $X_2$, verifying the claim.
\end{proof}
The trace embedding lemma is one of the most crucial connections
between sliceness properties of knots/links and exotic structures. The
version of this lemma for knots is rather well-known; here we discuss
a straightforward extension to links.
\begin{lemma}[Trace embedding lemma for links]
 \label{lemma:trace}
 A link $L$ in $S^3$ is smoothly slice in a possibly exotic $4$-disk $X$ if and only if $X(L^*)\hookrightarrow\widehat X$, where $\widehat X$ is the homotopy $4$-sphere obtained by attaching a $4$-handle to $X$, $L^*$ is the mirror image of $L$ and $X(L^*)$ is the $0$-trace of $L^*$.
\end{lemma}
\begin{proof}
 Suppose that $L$ is smoothly slice in $X$, which means that each
 component $L_i$ of $L$ bounds a properly embedded disk $D_i$ in $X$
 and $D_i\cap D_j=\emptyset$ for $i\neq j$. Take $\widehat X$ and one
 of its handle decompositions in the way that $\widehat
 X=D^4\cup_{S^3}X$. Hence, we can view $L^*$ 
 in $\partial D^4=S^3$ and then
 $D^4\cup\nu(D_1)\cup...\cup\nu(D_n)\cong X(L^*)\hookrightarrow
 X$. Note that since the $D_i$'s are disjoint, we can assume the same
 for the tubular neighborhoods $\nu(D_i)$'s.
 
 We now assume that $X(L^*)\hookrightarrow\widehat X$. We have that $\widehat X\cong
 X(L^*)\cup_fW$, where $W=\widehat X\setminus\mathring{X(L^*)}$ and $f\colon
 \partial X(L^*)\rightarrow \partial X(L^*)$ is the
 orientation-reversing diffeomorphism which acts as gluing
 map. Moreover, we can consider the handle decomposition on $X(L^*)$
 given by $D^4\cup\{\text{2-handles}\}$, where the 2-handles are
 attached along $L^*$ with framing 0.
 
 We see that the link $L^*$ sits in $S^3=\partial D^4$ and it bounds
 the cores of the 2-handles inside $X(L^*)$, which are embedded disks
 with boundary on $S^3$. Since $X\cong\widehat X\setminus\mathring
 D^4$, we obtain precisely that $L=(L^*)^*$ bounds a collection of
 disjoint properly embedded disks in $X$, hence $L$ is slice.
\end{proof}
\begin{remark}
There is also a locally flat version of the trace embedding lemma
which states that $L$ is topologically slice if and only if $X(L)$ is
a locally flat topological submanifold of $S^4$. Its proof proceeds in
the exact same way as the smooth case.
\end{remark}
With these preparations at hand, we are ready to verify that
the set of slice links will not help in detecting small exotic 4-disks.
\begin{prop}
 \label{prop:inclusion}
 If two possibly exotic $4$-disks are such that $X_1\hookrightarrow X_2$ then $\mathcal S_{X_1}\subset\mathcal S_{X_2}$.
\end{prop}
\begin{proof}
 We can apply Proposition \ref{prop:bs} to restate the trace embedding lemma as $L$ is smoothly slice in $X$ if and only if $X(L^*)\hookrightarrow X$. Therefore, we have $X(L^*)\hookrightarrow X_1\hookrightarrow X_2$ whenever $L$ is smoothly slice in $X_1$.
\end{proof}
\begin{proof}[Proof of Proposition~\ref{prop:smallexotic}]
 We have $D^4\hookrightarrow X\hookrightarrow D^4$, implying $\mathcal
 S_{D^4}\subset \mathcal S_X\subset\mathcal S_{D^4}$, which concludes
 the proof.
\end{proof}

\section{Traces of links and homotopy 4-spheres}
\label{sec:traces}
It is known \cite{Milnor} that every compact, simply connected, smooth $n$-manifold with $n\geq 5$ is geometrically simply connected; in other words, it admits a handle decomposition where there
are no 1-handles.  On the other hand, whether the same holds in
dimension four is still unknown, even for homotopy 4-spheres and
disks. Next we turn to the proof of our result on geometrically
simply connected possibly exotic 4-disks.

\begin{proof}[Proof of Theorem~\ref{thm:GSCPC}]
 First, we observe that $S$ is necessarily a homotopy 4-sphere. In
 fact, since it is closed and simply connected, by Freedman's
 classification result \cite{Freedman} we only need to check that
 $H_2(S;\Z)\cong\{0\}$, but this follows from the observation
 that \[b_2(S)=\chi(S)-2=\#|\text{2-handles}|-\#|\text{3-handles}|=n-n=0\:.\]
 If $S$ is exotic and $X(L)\cong\natural^nS^2\times D^2$ then we have
 a contradiction because $S$ would have a Kirby presentation
 consisting of just some $(2,3)$-cancelling pairs.  Let us assume now
 that $S\cong S^4$. Then we can apply Lemma \ref{lemma:bouquet} to
 prove that $X(L)$ is diffeomorphic to $\natural^n S^2\times D^2$ and
 this shows the other implication.
 
 Finally, given a geometrically simply connected homotopy 4-sphere, we
 take $L$ as the framed link which presents its 2-handles. Then $L$
 has as many components as there are 3-handles because of  Euler
 characteristic, and framings zero as $H_1(Y;\Z)\cong\Z^{|L|}$.
\end{proof}
This result leads us to examine exoticness on $S^2\times D^2$ as well.
\begin{thm}
 \label{thm:correspondence}
 There is a one-to-one correspondence between exotic $4$-spheres and exotic $S^2\times D^2$'s, up to diffeomorphism. In particular, the smooth $4$-dimensional Poincar\'e conjecture is equivalent to the existence of an exotic $S^2\times D^2$. 
\end{thm}
\begin{proof}
 Given an exotic $S^2\times D^2$ (say $X$), we obtain a homotopy 4-sphere $S$ by
 gluing $S^1\times D^3$ to $X$; and $S$ is exotic because of Lemma
 \ref{lemma:bouquet}. If $X\cong X'$ then obviously $S\cong S'$, while
 the converse is again true by Lemma~\ref{lemma:bouquet}.
 
 To see that this identification is surjective we start by an exotic
 homotopy 4-sphere $S$; then we consider a handle decomposition of $S$
 and we take $X$ as the manifold obtained by removing one 3-handle and
 one 4-handle from $S$.  The fact that such an $X$ is homeomorphic to
 $S^2\times D^2$ follows from the argument we used in the previous
 paragraph.
\end{proof}
Our goal now is to prove the converse of Proposition \ref{prop:inclusion} for geometrically simply connected exotic 4-disks.
\begin{prop}[Fake $(2,3)$-cancelling pair]
 \label{prop:fake}
 Given a compact, simply connected, smooth $4$-manifold $X$ with $\partial X=S^3$, admitting a handle decomposition given by $X=W\cup_{\#^kS^1\times S^2}\{k\:\:3-\text{handles}\}$, we have that $X\:\natural^kS^2\times D^2$ is diffeomorphic to $W$. 
\end{prop}
\begin{proof}
 It is a direct application of Lemma \ref{lemma:bouquet}. In fact, both $W$ and $X\:\natural^kS^2\times D^2$ become diffeomorphic to the closed and simply connected 4-manifold $\widehat X$ after attaching $\natural^kS^1\times D^3$.
\end{proof}
\begin{thm}
 \label{thm:slice_links}
 Two possibly exotic geometrically simply connected $4$-disks
 $X_1,X_2$ satisfy $X_1\hookrightarrow X_2$ if and only if $\mathcal
 S_{X_1}\subset\mathcal S_{X_2}$. In particular, a geometrically
 simply connected $4$-disk $X$ is small if and only if $\mathcal
 S_X=\mathcal S_{D^4}$.  
\end{thm}
\begin{proof}
 We only need to prove the 'if' implication because of Proposition
 \ref{prop:inclusion}. Let us consider an $n$-component link $L$ which
 presents the 2-handles of $X_1$; since $X(L)\hookrightarrow X_1$ one has that $L^*$
 is smoothly slice in $X_1$, and by assumption, also in $X_2$. Using
 the trace embedding lemma again, we then obtain that
 $X(L)\hookrightarrow X_2$.
 
 We saw that $X_2=X(L)\cup_{\#^nS^1\times S^2}Z$, where $Z=X_2\setminus\mathring{X(L)}$. We recall that Proposition \ref{prop:bs} assures us that we can take $X(L)\cap\nu(\partial X_2)=\emptyset$. 
 \begin{figure}[ht]
 \centering
 \def\svgwidth{10cm}
 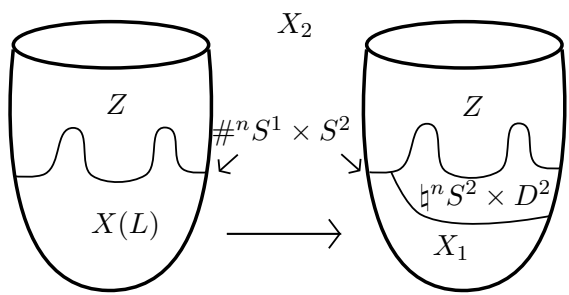   
 \caption{Since $X_1\:\natural^nS^2\times D^2\cong X(L)$, we can
   surger $X(L)$ from $X_2$ and then glue it back using the
   diffeomorphism.}
 \label{Inclusion}
 \end{figure}
 Now we use Proposition \ref{prop:fake} to claim that $X_1\:\natural^nS^2\times D^2\cong X(L)$; hence, we have $X_2=(X_1\:\natural^nS^2\times D^2)\cup_{\#^nS^1\times S^2}Z$ which means $X_1\hookrightarrow X_2$. See Figure \ref{Inclusion}.
\end{proof}
A few  observations are in place. The first is that the
latter statement necessarily requires links: in fact, in the proof we
used the fact that a geometrically simply connected exotic 4-disk has
a Kirby presentation which consists of a 0-framed link $L$. Such an
$L$ cannot be a knot because by Gabai's theorem (\cite{Gabai}) it would be the unknot. The second observation is
that, in the proof of Theorem \ref{thm:slice_links}, we did not
actually used that every link in $\mathcal S_{X_1}$ is contained in
$\mathcal S_{X_2}$, but just that a link which presents the 2-handles
of $X_1$ is. This is useful to prove the following corollary.

\begin{cor}
 \label{cor:large}
Every geometrically simply connected large exotic $4$-disk $X$ is
obtained by attaching $n$ $3$-handles on $X(J)$ for some $n$-component
link $J$, not smoothly slice in $D^4$ and such that
$S^3_{(0,...,0)}(J)\cong\#^nS^1\times S^2$.
\end{cor}

\begin{proof}
  We use the previous observation in the following way: assume the
  link $J$ is also smoothly slice in $D^4$, then we can mimic the
  proof of Theorem \ref{thm:slice_links} to show that
  $X\hookrightarrow D^4$. This is a contradiction because $X$ is
  large.
\end{proof}
We finally show an even deeper relation between exotic 4-disks and
exotic boundary sums of $S^2\times D^2$'s. First, we fix the notation
that two such manifolds $X_1$ and $X_2$ are called \emph{stably
  diffeomorphic} if $X_2\cong X_1\:\natural^nS^2\times D^2$ for some $n$.
\begin{thm}
 \label{thm:number}
 The number of geometrically simply connected exotic $4$-spheres, up to diffeomorphism, coincides with the sum of the numbers of exotic $\natural^iS^2\times D^2$ for $i\geq1$ which arise as $0$-traces of links in $S^3$, up to stable diffeomorphism. Equivalently, this is the number of diffeomorphism types of $S^2\times D^2$'s which are stably diffeomorphic to $0$-traces of links in $S^3$.
\end{thm}
In contrast, by Wall's theorem
\cite{Wall} (see also \cite{GS} for more details),
for every homotopy 4-sphere $S$ there is
an $N\geq0$ such that \[S\:\#^N(S^2\times S^2)\cong\#^N(S^2\times
S^2).\] 

The proof of Theorem \ref{thm:number} requires a preliminary lemma.
\begin{lemma}
 \label{lemma:stably}
 Every exotic $\natural^{n+1}S^2\times D^2$ for some $n\geq1$ is stably diffeomorphic to an exotic $S^2\times D^2$.
\end{lemma}
\begin{proof}
Let $X$ denote our exotic $\natural^{n+1}S^2\times D^2$. We construct
a manifold $X_0$ by attaching $\natural^{n+1}S^1\times D^3$ to $X$,
obtaining a homotopy 4-sphere $\widehat X$, and then removing
$S^1\times D^3$. By Lemma \ref{lemma:bouquet} there is a unique way to
do this. Moreover, we call $X'$ the possibly exotic 4-disk determined
by $\widehat X$ (as in Proposition \ref{prop:bs}).  Applying
Proposition \ref{prop:fake} twice, we have that
\[
X\cong X'\:\natural^{n+1}S^2\times D^2\hspace{2cm}
\text{ and }\hspace{2cm}X_0\cong X'\:\natural S^2\times D^2\:.
\]
  Hence, we conclude that $X\cong X_0\:\natural^nS^2\times D^2$ and
  $X_0$ is an exotic $S^2\times D^2$ because (by our hypothesis) $X$ was
  exotic.
\end{proof}
We can now move to the proof of Theorem~\ref{thm:number}.
\begin{proof}[Proof of Theorem \ref{thm:number}]
 Theorems \ref{thm:GSCPC} and \ref{thm:correspondence} provide the
 identification between stable diffeomorphism types of exotic boundary
 sums of $S^2\times D^2$ and diffeomorphism types of geometrically
 simply connected exotic 4-spheres.  The fact that this correspondence
 is well-defined, follows from Kirby calculus; moreover, we see that
 injectivity is a consequence of Lemma \ref{lemma:bouquet} while
 surjectivity follows from Theorem \ref{thm:GSCPC}.  To conclude, the last
 statement follows directly from Theorem \ref{thm:correspondence} and
 Lemma \ref{lemma:stably}.
\end{proof}
We obtain the following corollary.
\begin{cor}
 Suppose that $X$ is a possibly exotic $S^2\times D^2$. Then $X$ is
 geometrically simply connected if and only if it is stably
 diffeomorphic to the $0$-trace of a link in $S^3$.
\end{cor}
\begin{proof}
  We apply Theorem \ref{thm:number}. If $X$ is geometrically simply
  connected, then the same is true for the homotopy 4-sphere $S$
  obtained by gluing $S^1\times D^3$ on $\partial X$. Conversely, if $S$
  is geometrically simply connected then there is an $i$ such that
  $X\:\natural^iS^2\times D^2$ is diffeomorphic to the 0-trace of a
  link. We saw in the proof of Lemma \ref{lemma:stably} that then $X$
  has a handle decomposition without 1-handles.
\end{proof}
Note that no exotic $S^2\times D^2$ can be the 0-trace of a knot, but some stabilize to the 0-trace of a link in the case a geometrically simply connected exotic 4-disk exists. 

\section{Embeddings of simply connected 4-manifolds}
\label{sec:frames}
If a smooth 4-manifold $X$ does not admit a handle decomposition
without 1-handles, then it does not seem sufficient to consider the
set $\mathcal S_X$ of links smoothly slice in $X$ in order to
characterize $X$.  We proceed in a more general way: first, given a closed, simply connected, smooth 4-manifold $M$, by Lemma
\ref{lemma:bouquet} we have that the submanifold
$M_k=M\setminus\mathring{(\natural^kS^1\times D^3)}$ is well-defined
for every $k\geq0$.  Therefore, we can define the set
$\widehat{\mathcal S}_{M_k}$ of \emph{smoothly slice framed links} in
$M_k$ for every $k$. The elements of this set are $N$-component links
$\vec{L}\hookrightarrow\#^kS^1\times S^2$, equipped with a framing
$t_i$ for each component, such that the corresponding slice disk $D_i$
in $M_k$ has tubular neighborhood whose relative Euler number agrees
with $t_i$ for $i=1,...,N$. Note that $M_k\cong
M_0\:\natural^kS^2\times D^2$ for every $k\geq0$.

We then consider $\widehat{\mathcal S}_M=\displaystyle\bigcup_{k\geq0}\widehat{\mathcal S}_{M_k}$ for every closed, simply connected, smooth 4-manifold $M$. 
\begin{remark}
 Note that the set
$\mathcal S_X$, introduced earlier, coincides with
$\widehat{\mathcal S}_{M_0}=\widehat{\mathcal S}_X$ where $X$
is a possibly exotic $D^4$ and $M$ the homotopy 4-sphere
obtained by gluing a 4-handle to $X$. 
\end{remark}
We use a similar construction to extend the notion of the \emph{trace
  of a framed link} $\vec L$ in $\#^kS^1\times S^2$: in fact, since we
can view $\vec L$ as embedded in $\partial(\natural^k S^1\times D^3)$,
we define $X(\vec L)$ as the 4-manifold obtained by attaching
2-handles along $\vec L$, with the given framing.  The notion of the
trace of a framed link is well-defined, as diffeomorphic framed links
in $\#^kS^1\times S^2$ possess diffeomorphic traces. This follows
from a result of Laudenbach and Po\'enaru in \cite{LP} 
telling us that every self-diffeomorphism of $\#^kS^1\times S^2$ extends
to  $\natural^kS^1\times D^3$.

We recall that the mirror image $\vec L^*$ of a framed link $\vec L$
is the mirror image of $L$, equipped with
the framings of $\vec L$, after reversed their
signs.  We then have the following version of the trace embedding
lemma.
\begin{lemma}[Trace embedding lemma for framed links]
 Let us assume that $M_k$ is obtained from a $4$-manifold $M$ as explained before.  Then a framed link $\vec
 L\hookrightarrow\#^kS^1\times S^2$ is smoothly slice (as a framed
 link) in $M_k$ if and only if $X(\vec L^*)\hookrightarrow M$.
\end{lemma}
\begin{proof}
 The proof proceeds in the same way as the one of Lemma
 \ref{lemma:trace}. Suppose that $\vec L$ is smoothly slice in $M_k$
 for some $k\geq0$, with slice disks $D_1,...,D_n$. Take a handle
 decomposition of $M$ such that $M=\natural^kS^1\times
 D^3\cup_{\#^kS^1\times S^2}M_k$. Thus we can view $L^*$ as a link in
 $\partial(\natural^kS^1\times D^3)=\#^kS^1\times S^2$. The manifold
 $\natural^kS^1\times D^3\cup\nu(D_1)\cup...\cup\nu(D_n)$ is
 diffeomorphic to $X(\vec L^*)$, since each $\nu(D_i)$ can be seen as
 a 2-handle attached to $\natural^kS^1\times D^3$ with framing
 reversed with respect to the one of $\vec L$.
 
 We now assume that $X(\vec L^*)\hookrightarrow M$. Considering the
 4-manifold upside down we obtain that $L$ bounds a collection of
 mutually disjoint embedded disks in $M_k$; moreover, the tubular
 neighborhoods of these disks have relative Euler numbers which
 coincide with the framings of $\vec L$, because they equal the attaching
 framings of the 2-handles in $X(\vec L^*)$ with reversed signs.
\end{proof}
We can then prove the following generalization of Theorem \ref{thm:slice_links}.
\begin{thm}
 \label{thm:slice_links_k}
 Let us consider two closed, simply connected, smooth $4$-manifolds
 $M$ and $N$. Then $N=M\#M'$ if and only if $\widehat{\mathcal
   S}_{M}\subset\widehat{\mathcal S}_{N}$. In particular, one has
 $\widehat{\mathcal S}_S=\widehat{\mathcal S}_{S^4}$ if and only if
 $S$ is a small homotopy $4$-sphere.
\end{thm}
\begin{proof}
 If $\vec L\in\widehat{\mathcal S}_{M}$ then by the trace embedding
 lemma we have $X(\vec L^*)\hookrightarrow
 M\setminus\mathring{D^4}\hookrightarrow N$, implying  $\vec
 L\in\widehat{\mathcal S}_{N}$.
 
 Assume now $\widehat{\mathcal S}_{M}\subset\widehat{\mathcal
   S}_{N}$. If we call $\vec L$ the framed link which presents the
 2-handles in a Kirby diagram of $M$ then $\vec
 L\hookrightarrow\natural^kS^1\times D^3$ and $M=X(\vec
 L)\cup\natural^tS^1\times D^3$; thus $\vec L^*\in\widehat{\mathcal
   S}_{M}\subset\widehat{\mathcal S}_{N}$. Since the trace embedding
 lemma implies $X(\vec L)\hookrightarrow N$, we conclude as in the
 proof of Theorem \ref{thm:slice_links}.
 $\widehat{\mathcal S}_S=\widehat{\mathcal S}_{S^4}$ then implies that
 $S_0\subset S^4$, therefore $S$ is a small homotopy 4-sphere.
\end{proof}
A corollary of this result is that the smooth 4-dimensional Sch\"onflies theorem is equivalent to claim that $S^4$ is determined by its set of framed smoothly slice links. In addition, we can rewrite the smooth 4-dimensional Poincar\'e conjecture as follows.
\begin{thm}
 There are no exotic $4$-spheres if and only if for every
 $n$-component framed link $\vec L\hookrightarrow\#^kS^1\times S^2$
 with $n\geq2k$, such that $S^3(\vec L)\cong\#^{n-k}S^1\times S^2$ and
 $X(\vec L)$ is simply connected, one has $X(\vec
 L)\cong\natural^{n-k}S^2\times D^2$.
 
 Furthermore, there exists a large exotic $D^4$ if and only if we can find an $\vec L$ as before which is not smoothly slice in $\natural^kS^2\times D^2$.
\end{thm}
\begin{proof}
 Every homotopy 4-sphere has a handle decomposition with $k$
 1-handles, $n$ 2-handles and $n-k$ 3-handles. By possibly multiplying
 the corresponding Morse function with $(-1)$, we can assume that
 $n-k\geq k$.
 
 The fact that all exotic spheres are diffeomorphic to $S^4$ is equivalent to
 claim that the 4-manifold given by
 $D^4\cup\{\text{1-handles}\}\cup\{\text{2-handles}\}$ is always
 diffeomorphic to $\natural^{n-k}S^2\times D^2$, from Lemma
 \ref{lemma:bouquet}. It is then immediate to see that the trace of
 every framed link $\vec L\hookrightarrow\#^kS^1\times S^2$,
 satisfying the hypothesis, always gives rise to such a manifold.
 
 The second statement is then a consequence of the first one, Theorem
 \ref{thm:slice_links_k} and its proof.
\end{proof}
If $M$ has a Kirby presentation with exactly $k$ 1-handles, we do not
actually need to check the entire $\widehat{\mathcal S}_M$.
\begin{prop}
 \label{prop:min_one_handles}
 Let us assume that $M$ is a $4$-manifold as before and $k$ is the
 minimum number of $1$-handles in a Kirby diagram for $M$. Then we
 have that $\widehat{\mathcal S}_M$ is determined by
 $\widehat{\mathcal S}_{M_k}$, in the sense that if there is another manifold $N$ with $\widehat{\mathcal S}_{M_k}\subset\widehat{\mathcal S}_{N_k}$ then $\widehat{\mathcal S}_{M}\subset\widehat{\mathcal S}_{N}$.
\end{prop}
\begin{proof}
 Take a framed link in $\#^kS^1\times S^2$ that presents the 2-handles in a Kirby diagram for $M$. Applying the same argument in the proof of Theorem \ref{thm:slice_links_k} yields $M_0=M\setminus\mathring D^4\hookrightarrow N$, but then Theorem \ref{thm:slice_links_k} also leads to $\widehat{\mathcal S}_{M}\subset\widehat{\mathcal S}_{N}$.
\end{proof}
When $M$ and $N$ are such that $\widehat{\mathcal S}_{M}=\widehat{\mathcal S}_{N}$, from Theorem \ref{thm:slice_links_k} we have $N\cong M\#S$, where $S$ is a homotopy 4-sphere. Hence, we obtain Theorem \ref{thm:type}.
\begin{proof}[Proof of Theorem \ref{thm:type}]
 If $M\cong N$ then obviously one has $\widehat{\mathcal S}_{M}=\widehat{\mathcal S}_{N}$. Conversely, because of the observation above, when $\widehat{\mathcal S}_{M}=\widehat{\mathcal S}_{N}$ we have $N\cong M\#S$, but the manifold $S$ has to be diffeomorphic to $S^4$, since we are assuming that 4SPC holds.
\end{proof}

\section{High order traces and applications}
\label{sec:higherorder}
Inspired by the higher genus traces defined by Hayden and Piccirillo
in \cite{HP}, we generalize this concept to the link setting in the
way that it can be applied to give an alternative characterization of
the slice genus. We start by describing the two main constructions,
leaving the general case for later: the first 4-manifold, which here
we denote with $X^{g,1}(K)$, is the genus $g$ 2-handle attached
along the knot $K$ in $S^3$ (with framing $0$) appearing in \cite{HP};
while the second one consists of attaching a planar (genus zero)
2-handle with $\ell$ boundary component along an $\ell$-component link
in $S^3$, and we call it $X^{0,\ell}(L)$.

\paragraph*{High genus 2-handles} We recall that $X^{g,1}(K)$ is gotten by gluing $F_g\times D^2$, where $F_g=(\#^gT^2)\setminus\mathring{D^2}$ is the punctured connected sum of $g$ tori, together with an identification $f:\nu(K)\rightarrow\partial F_g\times D^2$ where $K\hookrightarrow S^3=\partial D^4$ is a knot.
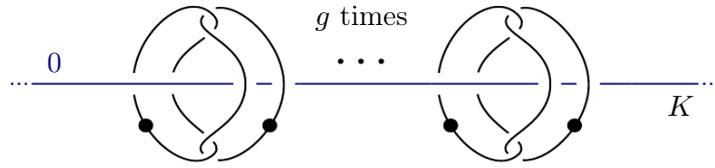
\begin{figure}[ht]
 \centering
 \def\svgwidth{10cm}
\begingroup%
  \makeatletter%
  \providecommand\color[2][]{%
    \errmessage{(Inkscape) Color is used for the text in Inkscape, but the package 'color.sty' is not loaded}%
    \renewcommand\color[2][]{}%
  }%
  \providecommand\transparent[1]{%
    \errmessage{(Inkscape) Transparency is used (non-zero) for the text in Inkscape, but the package 'transparent.sty' is not loaded}%
    \renewcommand\transparent[1]{}%
  }%
  \providecommand\rotatebox[2]{#2}%
  \newcommand*\fsize{\dimexpr\f@size pt\relax}%
  \newcommand*\lineheight[1]{\fontsize{\fsize}{#1\fsize}\selectfont}%
  \ifx\svgwidth\undefined%
    \setlength{\unitlength}{1537.51300049bp}%
    \ifx\svgscale\undefined%
      \relax%
    \else%
      \setlength{\unitlength}{\unitlength * \real{\svgscale}}%
    \fi%
  \else%
    \setlength{\unitlength}{\svgwidth}%
  \fi%
  \global\let\svgwidth\undefined%
  \global\let\svgscale\undefined%
  \makeatother%
  \begin{picture}(1,0.22453766)%
    \lineheight{1}%
    \setlength\tabcolsep{0pt}%
    \put(0,0){\includegraphics[width=\unitlength,page=1]{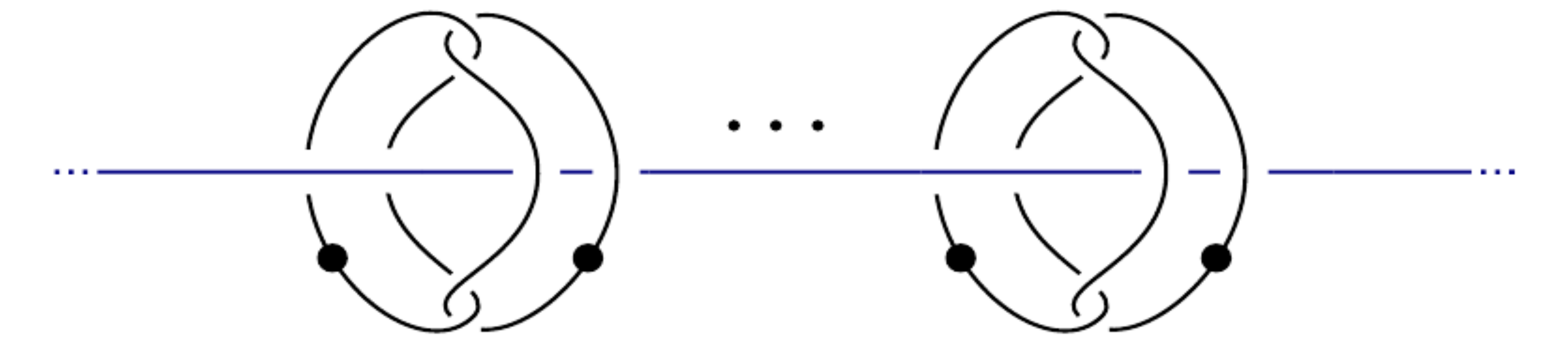}}%
    \put(0.08297871,0.12978724){\color[rgb]{0,0,0.50196078}\makebox(0,0)[lt]{\lineheight{1.25}\smash{\begin{tabular}[t]{l}0\end{tabular}}}}%
    \put(0.89858154,0.07163121){\color[rgb]{0,0,0}\makebox(0,0)[lt]{\lineheight{1.25}\smash{\begin{tabular}[t]{l}$K$\end{tabular}}}}%
    \put(0.43532186,0.19171333){\color[rgb]{0,0,0}\makebox(0,0)[lt]{\lineheight{1.25}\smash{\begin{tabular}[t]{l}$g$ times\end{tabular}}}}%
  \end{picture}%
\endgroup%
   
 \caption{A Kirby diagram of $X^{g,1}(K)$. The picture is taken from \cite{HP}.}
 \label{High}
\end{figure}
Obviously, in order to specify the diffeomorphism $f$ we also need to
fix a framing: we take the latter to be zero. The resulting 4-manifold
has a Kirby diagram as in Figure \ref{High}.

\paragraph*{Planar 2-handles} We now describe the construction of $X^{0,\ell}(L)$ and we show that it is closely related to the concept of knotification of a link $L$. We again start from $D^4$ and we want to glue $G_\ell\times D^2$, where $G_\ell=S^2\setminus\{\ell\text{ disks}\}$ is a 2-sphere with $\ell$ punctures, along $L$, an $\ell$-component link in $\partial
D^4$. The attaching region of the handle $G_\ell\times D^2$ is $\partial G_\ell\times D^2\cong A_1\sqcup...\sqcup A_\ell$ where $A_i\cong S^1\times D^2$ for each $i$; hence, this time we need $\ell$ gluing maps $f_i:\nu(L_i)\rightarrow A_i$ with framings $t_1,...,t_\ell$ such that \[t_1+...+t_\ell=-2\cdot\lk(L)\:,\] where here we take $\lk(L)=\displaystyle\sum_{i<j}\lk(L_i,L_j)$.
\begin{figure}[ht]
 \centering
 \def\svgwidth{7.5cm}
\begingroup%
  \makeatletter%
  \providecommand\color[2][]{%
    \errmessage{(Inkscape) Color is used for the text in Inkscape, but the package 'color.sty' is not loaded}%
    \renewcommand\color[2][]{}%
  }%
  \providecommand\transparent[1]{%
    \errmessage{(Inkscape) Transparency is used (non-zero) for the text in Inkscape, but the package 'transparent.sty' is not loaded}%
    \renewcommand\transparent[1]{}%
  }%
  \providecommand\rotatebox[2]{#2}%
  \newcommand*\fsize{\dimexpr\f@size pt\relax}%
  \newcommand*\lineheight[1]{\fontsize{\fsize}{#1\fsize}\selectfont}%
  \ifx\svgwidth\undefined%
    \setlength{\unitlength}{2432.94625097bp}%
    \ifx\svgscale\undefined%
      \relax%
    \else%
      \setlength{\unitlength}{\unitlength * \real{\svgscale}}%
    \fi%
  \else%
    \setlength{\unitlength}{\svgwidth}%
  \fi%
  \global\let\svgwidth\undefined%
  \global\let\svgscale\undefined%
  \makeatother%
  \begin{picture}(1,0.40966995)%
    \lineheight{1}%
    \setlength\tabcolsep{0pt}%
    \put(0,0){\includegraphics[width=\unitlength,page=1]{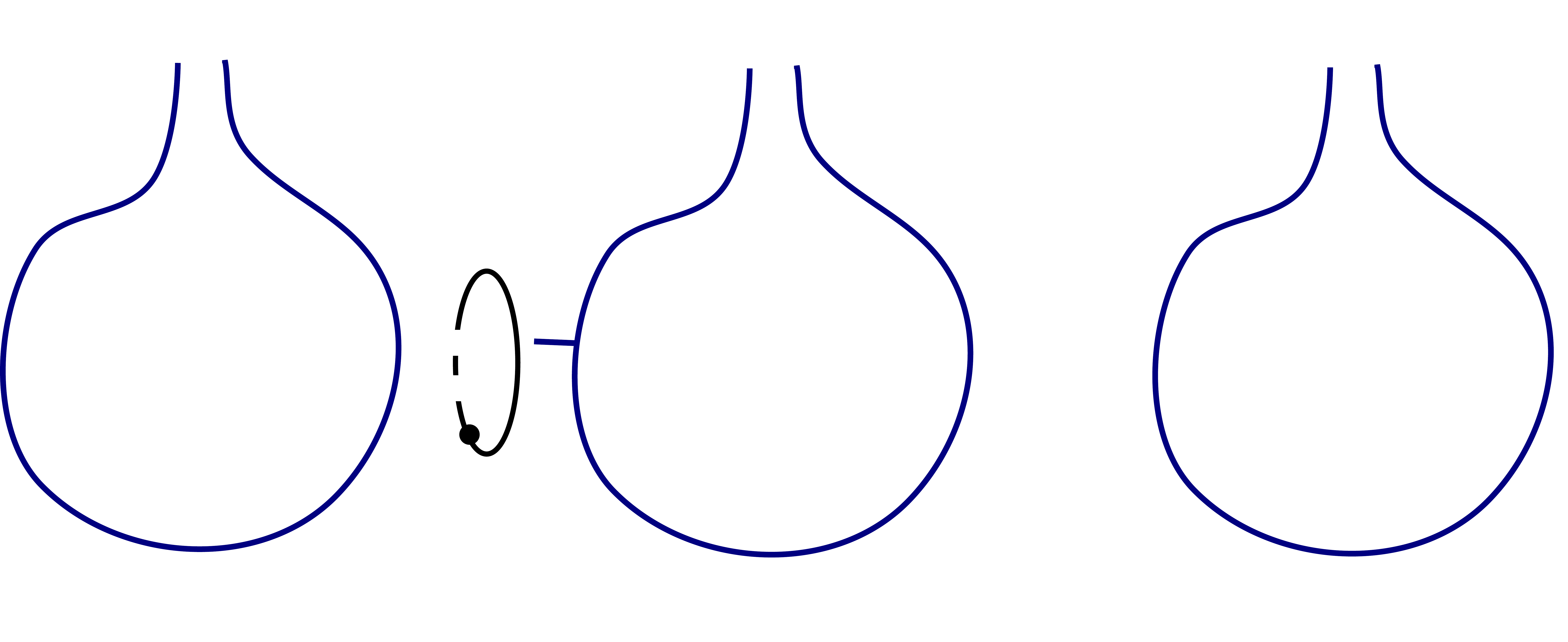}}%
    \put(0.53346153,0.02720591){\color[rgb]{0,0,0.50196078}\makebox(0,0)[lt]{\lineheight{1.25}\smash{\begin{tabular}[t]{l}0\end{tabular}}}}%
    \put(0,0){\includegraphics[width=\unitlength,page=2]{Planar2.pdf}}%
    \put(0.18026192,0.30712225){\color[rgb]{0,0,0}\makebox(0,0)[lt]{\lineheight{1.25}\smash{\begin{tabular}[t]{l}$K_L$\end{tabular}}}}%
  \end{picture}%
\endgroup%

 \hspace{1cm}
 \def\svgwidth{7.5cm}
\begingroup%
  \makeatletter%
  \providecommand\color[2][]{%
    \errmessage{(Inkscape) Color is used for the text in Inkscape, but the package 'color.sty' is not loaded}%
    \renewcommand\color[2][]{}%
  }%
  \providecommand\transparent[1]{%
    \errmessage{(Inkscape) Transparency is used (non-zero) for the text in Inkscape, but the package 'transparent.sty' is not loaded}%
    \renewcommand\transparent[1]{}%
  }%
  \providecommand\rotatebox[2]{#2}%
  \newcommand*\fsize{\dimexpr\f@size pt\relax}%
  \newcommand*\lineheight[1]{\fontsize{\fsize}{#1\fsize}\selectfont}%
  \ifx\svgwidth\undefined%
    \setlength{\unitlength}{2430.80733763bp}%
    \ifx\svgscale\undefined%
      \relax%
    \else%
      \setlength{\unitlength}{\unitlength * \real{\svgscale}}%
    \fi%
  \else%
    \setlength{\unitlength}{\svgwidth}%
  \fi%
  \global\let\svgwidth\undefined%
  \global\let\svgscale\undefined%
  \makeatother%
  \begin{picture}(1,0.4057034)%
    \lineheight{1}%
    \setlength\tabcolsep{0pt}%
    \put(0,0){\includegraphics[width=\unitlength,page=1]{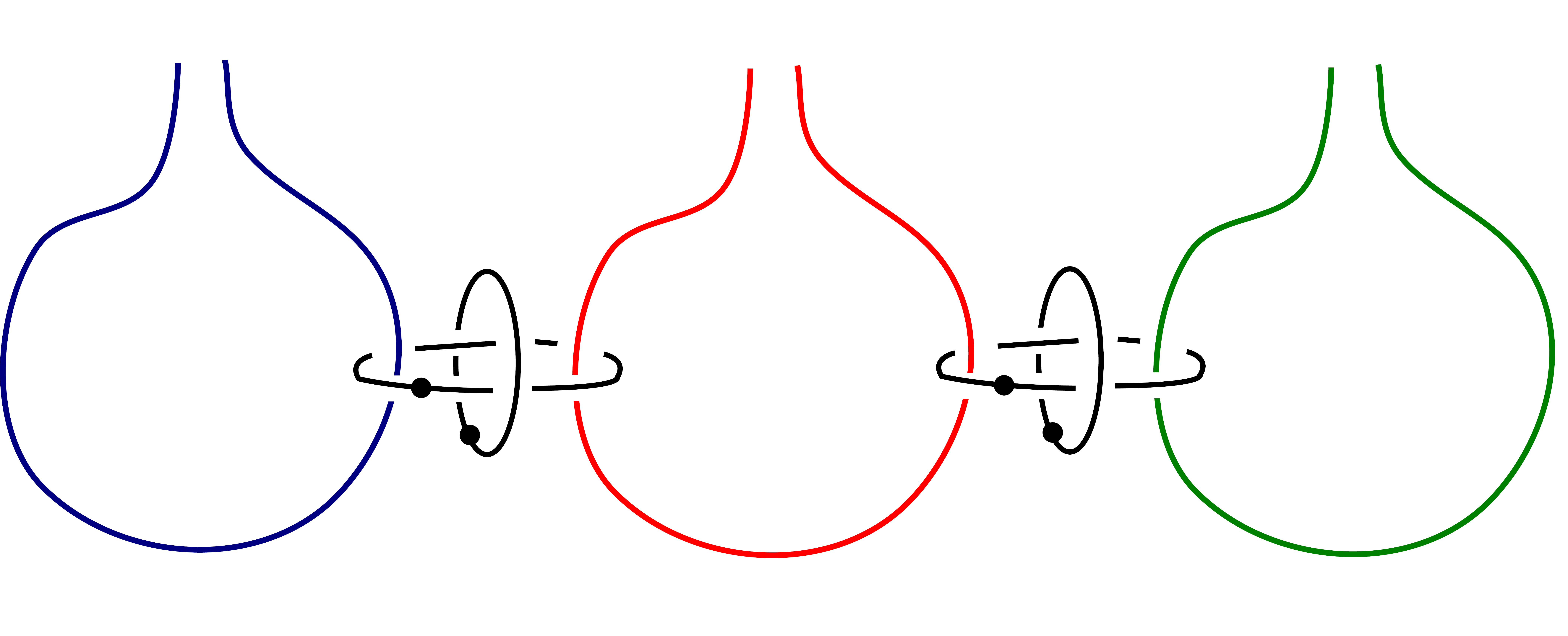}}%
    \put(0.0964788,0.01513175){\color[rgb]{0,0,0.50196078}\makebox(0,0)[lt]{\lineheight{1.25}\smash{\begin{tabular}[t]{l}$-2\lk(L)$\end{tabular}}}}%
    \put(0.53393093,0.02290282){\color[rgb]{1,0,0}\makebox(0,0)[lt]{\lineheight{1.25}\smash{\begin{tabular}[t]{l}0\end{tabular}}}}%
    \put(0.93215821,0.03017035){\color[rgb]{0,0.50196078,0}\makebox(0,0)[lt]{\lineheight{1.25}\smash{\begin{tabular}[t]{l}0\end{tabular}}}}%
    \put(0,0){\includegraphics[width=\unitlength,page=2]{Planar.pdf}}%
    \put(0.1952304,0.29195806){\color[rgb]{0,0,0}\makebox(0,0)[lt]{\lineheight{1.25}\smash{\begin{tabular}[t]{l}$L_1$\end{tabular}}}}%
    \put(0.55677603,0.29135866){\color[rgb]{0,0,0}\makebox(0,0)[lt]{\lineheight{1.25}\smash{\begin{tabular}[t]{l}$L_2$\end{tabular}}}}%
    \put(0.89863768,0.23582156){\color[rgb]{0,0,0}\makebox(0,0)[lt]{\lineheight{1.25}\smash{\begin{tabular}[t]{l}$L_3$\end{tabular}}}}%
  \end{picture}%
\endgroup%
   
 \caption{The 0-trace of the knot $K_L$, the knotification of $L$, seen on the boundary of $\natural^{\ell-1} S^1\times D^3$ (left) and an equivalent Kirby diagram of it (right).}
 \label{Planar}
\end{figure}

It is important to observe that the core of the handle (the surface $G_\ell\times\{0\}$) comes with an orientation that induces coherent orientations on each attaching sphere: such orientations have to agree with the ones of the components of $L$; this means that the manifold $X^{0,\ell}(L)$ depends on the relative orientation of $L$. 
Note that $X^{0,1}(K)$ is the 0-trace of the knot $K$.

We recall that the \emph{knotification} of an $\ell$-component link $L$ with $\ell>1$ is the knot $K_L\hookrightarrow\#^{\ell-1}S^1\times S^2$, which is shown on the left in Figure \ref{Planar}, obtained by adding $\ell-1$ oriented bands between the components of $L$, realizing $\ell-1$ merge moves; and then doing the same number of 0-surgeries along the boundaries of small disks, each with a unique ribbon intersection with one of the bands. Note that $K_L$ is null-homologous by construction.

Ozsv\'ath and Szab\'o proved in \cite{OSz} that this operation is well-defined; in other words, the diffeomorphism type of the knot $K_L$ is independent of the choice of the bands we use to perform the merge moves. For more details about knotification of links see also \cite{Kuzbary}.
\begin{prop}
 \label{prop:knotification}
 The $4$-manifold $X^{0,\ell}(L)$ for $\ell>1$ is diffeomorphic to $X(K_L)$, the trace of the $0$-framed knot $K_L\hookrightarrow\#^{\ell-1}S^1\times S^2$.
\end{prop}
\begin{proof}
 We split the handle attachment in two parts. We start by attaching the grey band in Figure \ref{Grey}, which means gluing $I\times I\times D^2$ along $\partial I\times I\times D^2$ with framings zero; this procedure needs to be repeated $\ell-1$ times. This requires the choice of $\ell-1$ pairs of oriented arcs $\alpha_i$ and $\beta_i$ inside the two components of $L$ where we are performing the $\ell-1$ merge moves; one has $\nu(\alpha_i)\cong\nu(\beta_i)\cong I\times D^2$ as attaching regions. 
 As explained before, the orientations need to agree accordingly. 
 
 The 4-manifold obtained by this bridge construction is exactly $\natural^{\ell-1}S^1\times D^3$; and in order to complete the original handle attachment we glue a standard 4-dimensional 2-handle $\overline G_\ell\times D^2$ along a framed knot $J$ in $\#^{\ell-1}S^1\times S^2=\partial(\natural^{\ell-1}S^1\times D^3)$. 
 \begin{figure}[ht]
 \centering
 \def\svgwidth{8cm}
 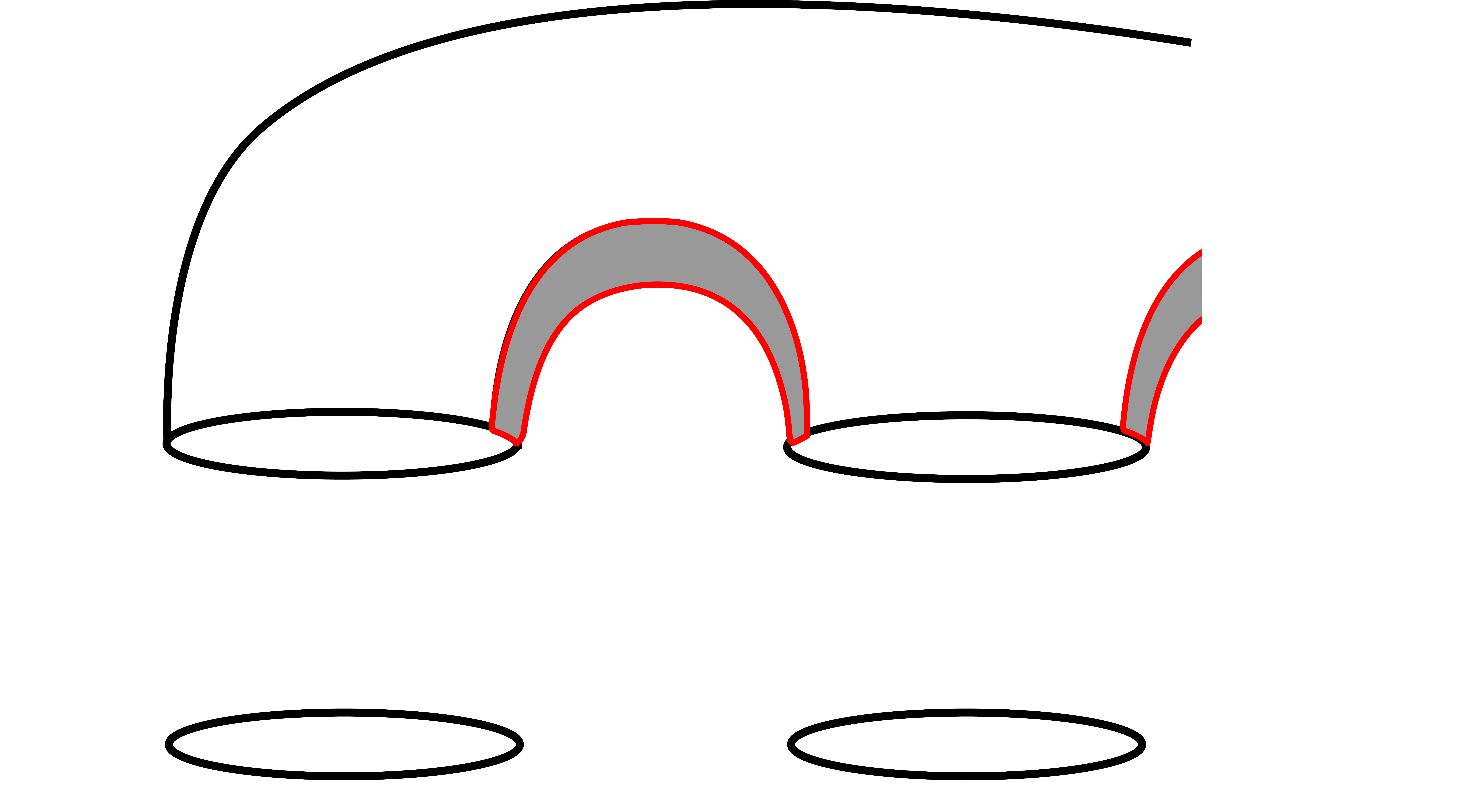
 \caption{Attaching a planar 2-handle along an $\ell$-component link $L$. The surface $\overline G_\ell$ is $G_\ell$ without the grey bands.}
 \label{Grey}
 \end{figure}
 The knot $J$ consists of the arcs
 $L\setminus(\alpha_1\cup...\cup\alpha_{\ell-1}\cup\beta_1\cup...\cup_{\ell-1})$
 joined by the $\ell-1$ pairs of arcs $a_i$ and $b_i$, corresponding
 to $I\times\partial I\times D^2$; hence, we have that $J$ is the
 knotification of $L$ since its construction matches the one of $K_L$
 in \cite{OSz}. In addition, its framing $t$ is determined as follows: \[t=\overline e_J(\overline G_\ell)=\overline e_L(G_\ell)=t_1+...+t_\ell+2\cdot\lk(L)=0\:,\] where $\overline e$ is the relative normal Euler number. 
\end{proof}
\paragraph*{High order 2-handles} Mixing the constructions explained in the previous two paragraphs gives rise to what we are going to call high order traces of a link. Let us start from a partition of an $\ell$-component link $L$ in $S^3$; we call $\mathcal P=\{P_1,...,P_k\}$ a \textbf{weighted partition} of $L$ if each sublink $P_i\in\mathcal P$, whose number of components is $\ell_i$, has an integer $g_i$ associated to it, and $g_i\geq0$ for every $i=1,...,k$.

We denote by \textbf{high order $0$-trace} of $L$ \textbf{with partition} $\mathcal P$ the 4-manifold $X^{\mathcal P}(L)$ obtained in the following way: we begin the attachment of 0-framed planar 2-handles along each sublink $P_i$, obtaining a 0-framed link with components $K_{P_1},...,K_{P_k}$ on the boundary of $\natural^{\ell-k}S^1\times D^3$; then we attach a 0-framed genus $g_i$ 2-handle along $K_{P_i}$ for $i=1,...,k$. The resulting 4-manifold is well-defined because the ordering of the gluings is unimportant; moreover, in the case when $k=1$ we just denote the trace with $X^{g,\ell}(L)$. A Kirby diagram for $X^{\mathcal P}(L)$ can be easily produced by combining the ones in Figures \ref{High} and \ref{Planar}.
\paragraph*{Trace embedding lemma} In accordance with other results in the previous section, here we prove a third version of the trace embedding lemma. We recall that if $X$ is a compact smooth 4-manifold and $\partial X=S^3$ then $\widehat X=X\cup\{4-\text{handle}\}$; in addition, when $L$ bounds a compact oriented surface $\Sigma$, such that all the connected components of $\Sigma$ have boundary in $L$, we say that $\mathcal P_\Sigma$ is the weighted partition on $L$ \emph{induced by} $\Sigma$ in the natural way.
\begin{lemma}[Trace embedding lemma for high order traces]
 A link $L$ in $S^3$ bounds a compact, oriented, smooth surface $\Sigma$, properly embedded in a possibly exotic $4$-disk $X$, if and only if $X^{\mathcal P_\Sigma}(L^*)$ is smoothly embedded in $\widehat X$. 
\end{lemma}
\begin{proof}
 For the first implication we assume that $X^{\mathcal P_\Sigma}(L^*)\hookrightarrow\widehat X$. Hence, we proceed as in other versions of the lemma and we write $\widehat X=D^4\cup_{S^3}X$ with $L^*\hookrightarrow S^3=\partial D^4$ and the high order 2-handles inside $X$. On each sublink $P^*_1,...,P^*_k$ of $L^*$ in the weighted partition $\mathcal P_\Sigma$, the core of the corresponding high order 2-handle, whose attaching sphere is precisely $P_i^*$, is diffeomorophic to the surface $\Sigma_i=F_{g_i}\#G_{{\ell_i}-1}$. Since all the high order 2-handles we attached are disjoint, we have that $\Sigma=\Sigma_1\cup...\cup\Sigma_k\hookrightarrow X$ and its boundary is the link $(L^*)^*=L$. 
 
 We now prove the converse. Let us consider $\widehat X=D^4\cup_{S^3}X$ and take $L^*\hookrightarrow S^3=\partial D^4$. We have that each sublink $P_i^*$ of $L^*$ bounds a connected component $\Sigma_i$ of $\Sigma$, with genus $g_i$ and $\ell_i$ boundary components according to $\mathcal P_\Sigma$, for $i=1,...,k$. Since we know that $H_2(X,\partial X;\Z)\cong\{0\}$, we obtain that each $\Sigma_i$ bounds an embedded 3-homology class in $X$; thus showing that the relative Euler number $\overline e$ of $\nu(\Sigma_i)$ is zero.
 We can then conclude that \[t_1+...+t_{\ell_i}+2\cdot\lk(P_i)=\overline e_{P_i}(\Sigma_i)=0\] and then $\Sigma_i$ is glued with framing $-2\cdot\lk(P_i)$ to $D^4$ for any $i$.
 
 We have proved that $(D^4\cup\nu(\Sigma_i))\cong
 X^{g_i,\ell_i}(P_i^*)$ for every $i=1,...,k$. Hence, since the
 4-manifold $X^{\mathcal P_\Sigma}(L^*)$ is constructed by
 performing $k$ consecutive high order 2-handle attachments along the
 $P_i^*$'s, the latter just appears to be diffeomorphic to
 $D^4\cup\nu(\Sigma)$ which is embedded in $\widehat X$.
\end{proof}
\begin{remark}
 As for Lemma \ref{lemma:trace}, the same proof also shows that $L$ bounds a compact oriented $\Sigma$, properly locally flat embedded in $D^4$, if and only if $X^{\mathcal P_\Sigma}(L)$ is a locally flat submanifold of $S^4$.
\end{remark}

\subsection{Slice genera of links in homotopy 4-spheres}
The slice genus $g_4$ for a knot $K$ has a unique natural definition:
the minimal genus of a compact, oriented, connected surface $F$,
properly and smoothly embedded in $D^4$, such that $K=\partial F$. This is not the
case for links; there are many versions of slice genera of a link
$L$. In this paper we are going to recall the most important ones, but
focusing only on the two that are studied more often.

Consider an $\ell$-component link. We denote by $g_4(L)$ what it is
usually called the \emph{slice genus} of $L$, which has exactly the
same definition of $g_4$ for knots. Moreover, we define $g_4^X(L)$ the
same invariant, but where surfaces are taken in an exotic 4-disk $X$
instead of $D^4$. We recall that $L$ is said to bound a planar smooth surface in $X$ when $g_4^X(L)=0$.

In addition, it follows from standard results that every
null-homologous knot $J$ in $\#^nS^1\times S^2$ bounds a
null-homologous properly embedded surface in $\natural^nS^2\times
D^2$, for every smooth structure and $n\geq0$. Hence, the knot $J$ is
smoothly slice, in a possibly exotic $\natural^nS^2\times D^2$, if and
only if it is $H$-slice; and in this case the only possible value for
the framing is zero. We recall that a knot $K$ is \emph{$H$-slice}, in
a compact 4-manifold $W$ with boundary, when $K$ bounds a
null-homologous properly embedded disk in $W$.
\begin{prop}
 \label{prop:planar}
 An $\ell$-component link $L$ in $S^3$ bounds a planar smooth surface in $X$ if and only if its knotification $K_L$ is smoothly $H$-slice in $X\:\natural^\ell S^2\times D^2$.
\end{prop}
\begin{proof}
 Using the trace embedding lemma for high order traces, we have that
 $L$ bounds a planar smooth surface in $X$ if and only if
 $X^{0,\ell}(L)\hookrightarrow\widehat X$. According to Proposition
 \ref{prop:knotification}, one also has $X^{0,\ell}(L)\cong
 X(K_L)$ and, since $K_L$ is null-homologous, according to what we
 said before we just need to apply the trace embedding lemma for
 framed links.
\end{proof}
The proof of the 'only if' implication in the latter proposition already
appeared in \cite{Kuzbary}.  We now recall the definition of the
\emph{maximal $4$-dimensional Euler characteristic} of $L$ as the
maximum of $\chi(\Sigma)$, where $\Sigma$ is a compact oriented surface,
properly and smoothly embedded in $X$, such that $L=\partial\Sigma$ and every
connected component of $\Sigma$ has boundary in $L$. We denote this
invariant with $\chi_4^X(L)$; note that the difference between the
surfaces $F$ and $\Sigma$ is that the second one is not necessarily
connected. $\chi_4^X(L)$ takes integer values which can be at most
$\ell$, with equality if and only if $L$ is smoothly slice in $X$.
\begin{remark}
 Another version of the slice genus which sometimes appears in the
 literature is $g_4^*$, which is defined as the minimal genus of
 $\Sigma$ as before, but together with the condition that each
 connected component of $\Sigma$ has exactly one boundary
 component. The invariant $g_4^*$ is less used because it can only be
 defined for links with zero linking matrix, nonetheless Theorem
 \ref{thm:slice} holds for it too.
 
 Furthermore, we also mention the fact that some authors prefer to renormalize $\chi_4(L)$ as follows: \[2G_4(L)-\ell=-\chi_4(L)\:,\] because in this way one has $G_4(L)\geq0$ as for $g_4(L)$. With the latter convention, we have $G_4(L)=0$ if and only if $L$ is smoothly slice.
\end{remark}
In Proposition \ref{prop:smallexotic} we showed that if $X$ is a small exotic $4$-disk then $\mathcal S_X=\mathcal S_{D^4}$, in other words the sets of slice links in $X$ and $D^4$ coincide. Now we prove that the same is true for the slice genus and the maximal 4-dimensional Euler characteristic.
\begin{proof}[Proof of Theorem \ref{thm:slice}]
 The strategy of the proof is to argue that if $X_1$ and $X_2$ are two possibly exotic 4-disks such that $X_1\hookrightarrow X_2$ and $X^{\mathcal P_{\Sigma}}(L)\hookrightarrow X_1$ then obviously one has $X^{\mathcal P_{\Sigma}}(L)\hookrightarrow X_2$ for every weighted partition of $L$. Since by assumption we have $D^4\hookrightarrow X\hookrightarrow D^4$, it follows that $X^{\mathcal P_{\Sigma}}(L)\hookrightarrow X$ if and only if $X^{\mathcal P_{\Sigma}}(L)\hookrightarrow D^4$; hence, by applying the trace embedding lemma for high order traces we obtain that $L$ bounds a surface $\Sigma$ in $X$ if and only if the same happens in $D^4$.
\end{proof}

\subsection{Examples of non-slice null-homologous knots in \texorpdfstring{$S^1\times S^2$}{S1xS2}}
In the last subsection of the paper we want to apply Proposition \ref{prop:planar} to give an example of an infinite family of null-homologous knots in $S^1\times S^2$ which are not smoothly slice in $S^2\times D^2$. 
\begin{figure}[ht]
 \centering
 \includegraphics[width=6cm]{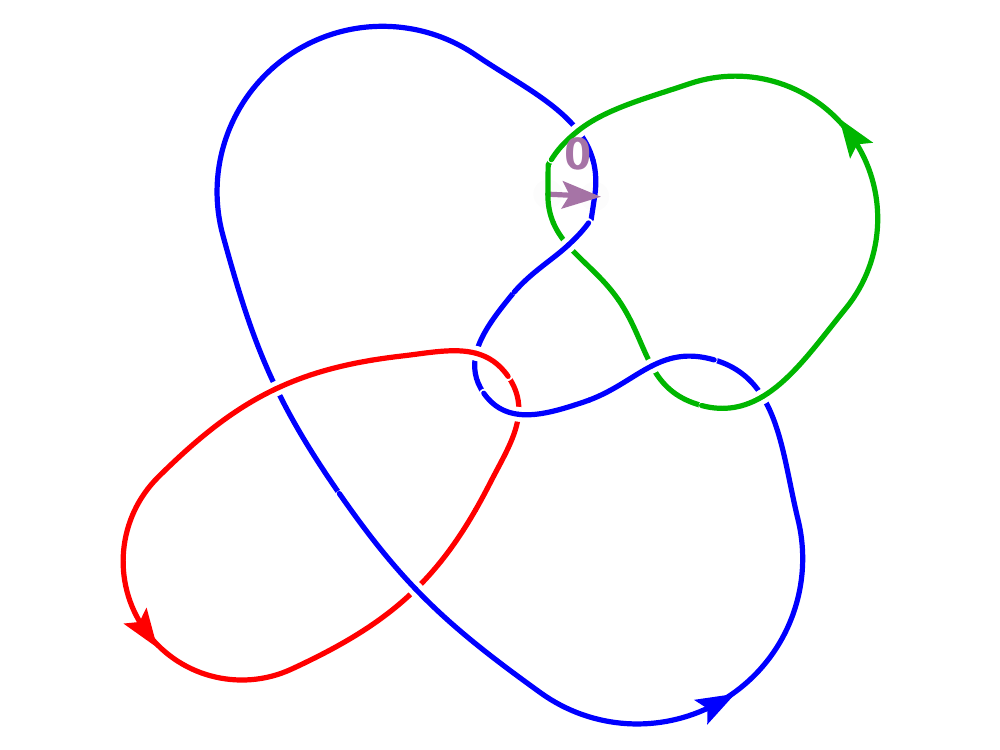}
 \hspace{1cm}
 \includegraphics[width=6cm]{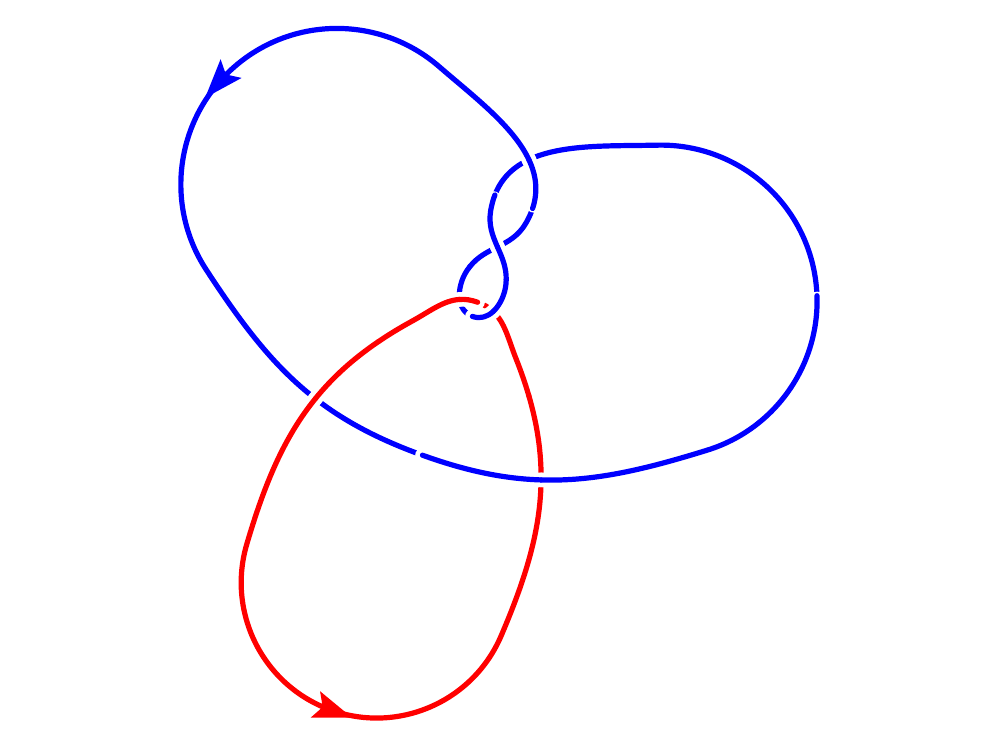}
 \caption{A diagram for the Borromean link $B$ (left) and the Whitehead link $W$ (right). An untwisted band is highlighted: this represents a merge move connecting $B$ to $W$.}
 \label{B-W}
\end{figure}
As already observed by Kuzbary in \cite{Kuzbary}, to obstruct that a knot $K$ in $S^1\times S^2$ is not smoothly slice in $S^2\times D^2$ is not enough to show that $K$ is not concordant to the unknot; in fact, the knot $K_{L_1}$ in Figure \ref{L_n} is $H$-slice in $S^2\times D^2$, because it is the knotification of the Hopf link which bounds a smooth annulus in $S^3$, but in \cite{Kuzbary} it is proved to not be concordant to the unknot.

We start by a special case which involves two of the most studied links in $S^3$: the Borromean link $B$ and the Whitehead link $W$; a diagram for these links appears in Figure \ref{B-W}. Conway and Orson proved in \cite{CO} that $g_4(B)=1$ with every relative orientation; this also shows that $g_4(W)=1$: in fact, in the case $W$ was the boundary of a smooth annulus in $D^4$, we could build a planar smooth surface bounded by $B$ using the merge move pictured in Figure \ref{B-W}.
\begin{figure}[ht]
 \centering
 \includegraphics[width=7cm]{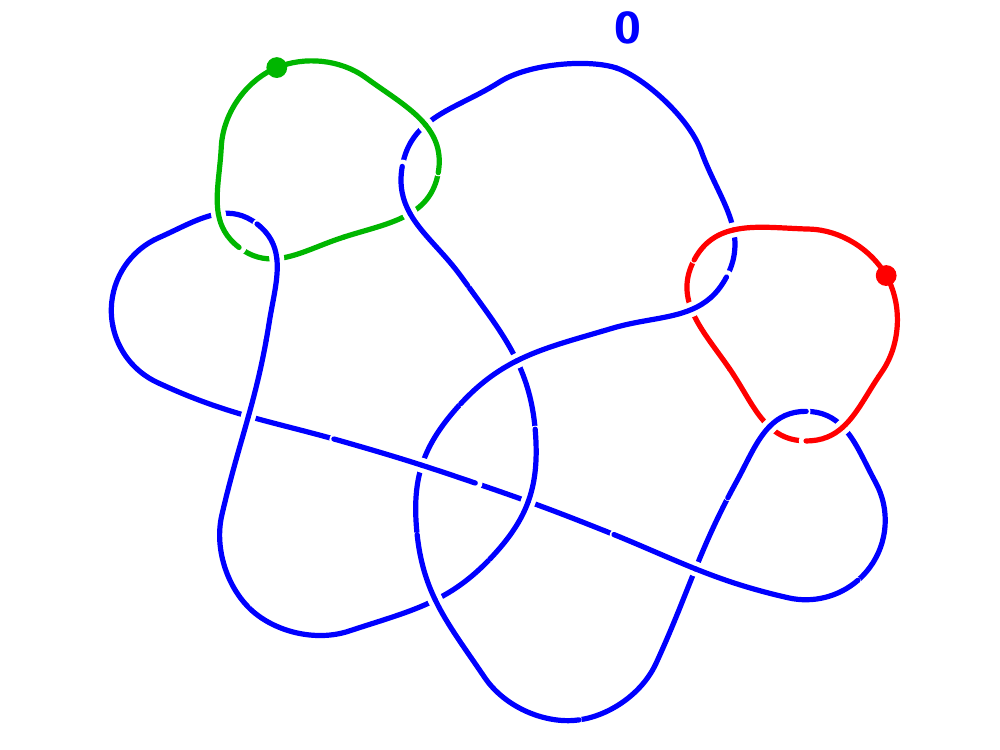}
 \caption{The $0$-trace of $K_B\hookrightarrow\#^2S^2\times D^2$, the knotification of the Borromean link, with one of the possible relative orienations.}
 \label{K_B}
\end{figure}
Proposition \ref{prop:knotification} gives that the corresponding knotifications $K_B$ and $K_W=K_{L_2}$ are not smoothly slice in $\#^2S^2\times D^2$ and $S^2\times D^2$ respectively; such knots appear in Figures \ref{K_B} and \ref{L_n}.
\begin{prop}
 The knots $K_{L_n}$ in \emph{Figure \ref{L_n}} for $n\leq0$ form an infinite family of null-homologous knots in $S^1\times S^2$ which are not smoothly slice in $S^2\times D^2$. 
\end{prop}
\begin{figure}[ht]
 \centering
 \def\svgwidth{6cm}
\begingroup%
  \makeatletter%
  \providecommand\color[2][]{%
    \errmessage{(Inkscape) Color is used for the text in Inkscape, but the package 'color.sty' is not loaded}%
    \renewcommand\color[2][]{}%
  }%
  \providecommand\transparent[1]{%
    \errmessage{(Inkscape) Transparency is used (non-zero) for the text in Inkscape, but the package 'transparent.sty' is not loaded}%
    \renewcommand\transparent[1]{}%
  }%
  \providecommand\rotatebox[2]{#2}%
  \newcommand*\fsize{\dimexpr\f@size pt\relax}%
  \newcommand*\lineheight[1]{\fontsize{\fsize}{#1\fsize}\selectfont}%
  \ifx\svgwidth\undefined%
    \setlength{\unitlength}{240.79947699bp}%
    \ifx\svgscale\undefined%
      \relax%
    \else%
      \setlength{\unitlength}{\unitlength * \real{\svgscale}}%
    \fi%
  \else%
    \setlength{\unitlength}{\svgwidth}%
  \fi%
  \global\let\svgwidth\undefined%
  \global\let\svgscale\undefined%
  \makeatother%
  \begin{picture}(1,0.89701189)%
    \lineheight{1}%
    \setlength\tabcolsep{0pt}%
    \put(0,0){\includegraphics[width=\unitlength,page=1]{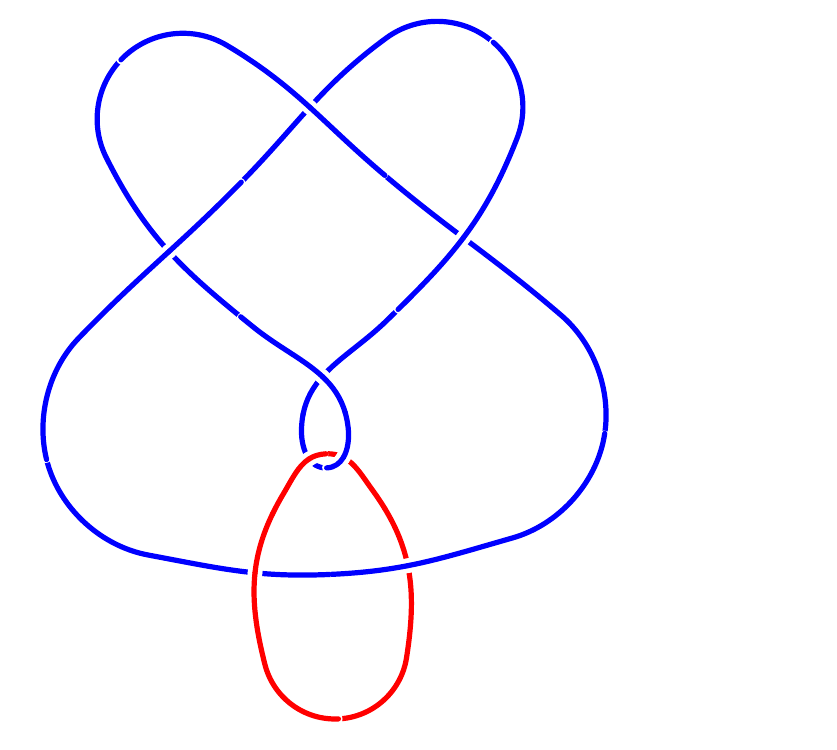}}%
    \put(0.00833974,0.32351863){\color[rgb]{0,0,1}\makebox(0,0)[lt]{\lineheight{1.25}\smash{\begin{tabular}[t]{l}0\end{tabular}}}}%
    \put(0.00584804,0.32351863){\color[rgb]{0,0,1}\makebox(0,0)[lt]{\lineheight{1.25}\smash{\begin{tabular}[t]{l}0\end{tabular}}}}%
    \put(0,0){\includegraphics[width=\unitlength,page=2]{K_L_n.pdf}}%
    \put(0.66916238,0.56844501){\color[rgb]{0,0,0}\makebox(0,0)[lt]{\lineheight{1.25}\smash{\begin{tabular}[t]{l}$K_{L_n}$\end{tabular}}}}%
    \put(0.36700818,0.37976993){\color[rgb]{0,0,0}\makebox(0,0)[lt]{\lineheight{1.25}\smash{\begin{tabular}[t]{l}$n$\end{tabular}}}}%
  \end{picture}%
\endgroup%
   
 \hspace{1cm}
 \def\svgwidth{5.7cm}
\begingroup%
  \makeatletter%
  \providecommand\color[2][]{%
    \errmessage{(Inkscape) Color is used for the text in Inkscape, but the package 'color.sty' is not loaded}%
    \renewcommand\color[2][]{}%
  }%
  \providecommand\transparent[1]{%
    \errmessage{(Inkscape) Transparency is used (non-zero) for the text in Inkscape, but the package 'transparent.sty' is not loaded}%
    \renewcommand\transparent[1]{}%
  }%
  \providecommand\rotatebox[2]{#2}%
  \newcommand*\fsize{\dimexpr\f@size pt\relax}%
  \newcommand*\lineheight[1]{\fontsize{\fsize}{#1\fsize}\selectfont}%
  \ifx\svgwidth\undefined%
    \setlength{\unitlength}{249.6338489bp}%
    \ifx\svgscale\undefined%
      \relax%
    \else%
      \setlength{\unitlength}{\unitlength * \real{\svgscale}}%
    \fi%
  \else%
    \setlength{\unitlength}{\svgwidth}%
  \fi%
  \global\let\svgwidth\undefined%
  \global\let\svgscale\undefined%
  \makeatother%
  \begin{picture}(1,0.86526723)%
    \lineheight{1}%
    \setlength\tabcolsep{0pt}%
    \put(0,0){\includegraphics[width=\unitlength,page=1]{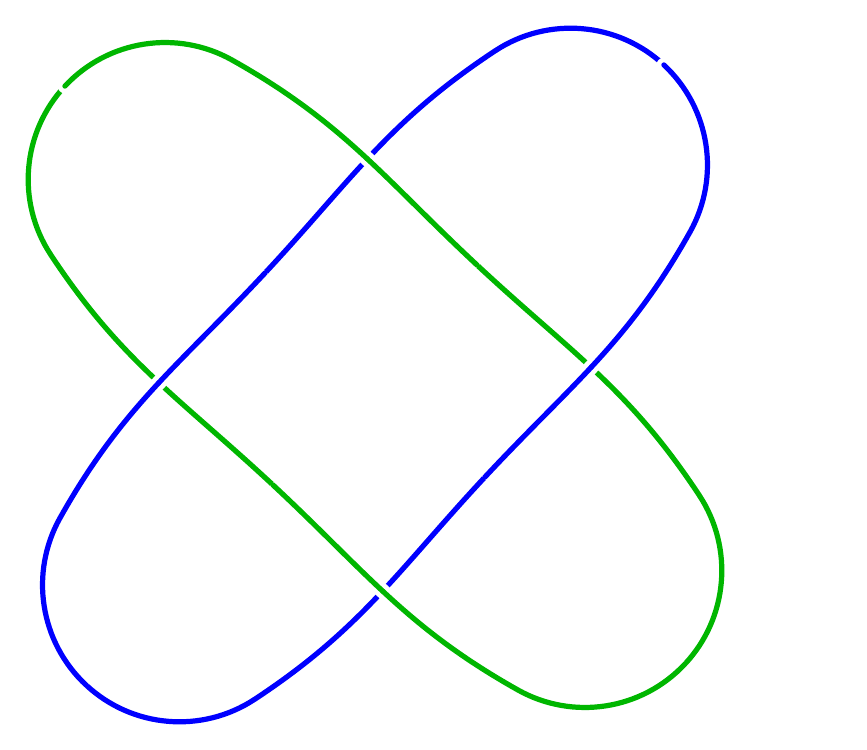}}%
    \put(0.05558221,0.02909353){\color[rgb]{0,0,1}\makebox(0,0)[lt]{\lineheight{1.25}\smash{\begin{tabular}[t]{l}0\end{tabular}}}}%
    \put(0.05317869,0.02909353){\color[rgb]{0,0,1}\makebox(0,0)[lt]{\lineheight{1.25}\smash{\begin{tabular}[t]{l}0\end{tabular}}}}%
    \put(0.03094256,0.77671573){\color[rgb]{0,0.71372549,0}\makebox(0,0)[lt]{\lineheight{1.25}\smash{\begin{tabular}[t]{l}0\end{tabular}}}}%
    \put(0.02853904,0.77671573){\color[rgb]{0,0.71372549,0}\makebox(0,0)[lt]{\lineheight{1.25}\smash{\begin{tabular}[t]{l}0\end{tabular}}}}%
    \put(0,0){\includegraphics[width=\unitlength,page=2]{L_n.pdf}}%
    \put(0.82686557,0.31347395){\color[rgb]{0,0,0}\makebox(0,0)[lt]{\lineheight{1.25}\smash{\begin{tabular}[t]{l}$L_n$\end{tabular}}}}%
    \put(0.42433047,0.08908011){\color[rgb]{0,0,0}\makebox(0,0)[lt]{\lineheight{1.25}\smash{\begin{tabular}[t]{l}$n$\end{tabular}}}}%
  \end{picture}%
\endgroup%
 
 \caption{The link $L_n$ (right) and the 0-trace of its knotification $K_{L_n}$ (left). Here $n$ is the number of full twists in the corresponding tangle.}
 \label{L_n}
\end{figure}
\begin{proof}
 By construction the knot $K_{L_n}$ is the knotification of the link $L_n$ which is shown on the right of Figure \ref{L_n}. When $n\leq0$ we see that $L_n$ is a non-split, alternating, 2-component link: we can then easily compute the $\tau$-invariant, as in \cite{Cavallo}, from the signature; and we get $\tau(L_n)=2$ for every $n\leq0$. Therefore, the slice genus bound from \cite{Cavallo} yields \[2=\tau(L_n)\leq g_4(L_n)+\ell-1=g_4(L_n)+1\] which means $g_4(L_n)\geq1$. Since $L_n$ never bounds a planar smooth surface in $D^4$ for $n\leq0$, we can apply Proposition \ref{prop:knotification} to argue that $K_{L_n}$ is not smoothly slice in $S^2\times D^2$.
\end{proof} 
It is actually possible to prove that $g_4(L_n)=1$ for each $n\leq0$, because $L_n$ can be turned into a positive trefoil after a merge move, and this, in light of the trace embedding lemma for high order traces, tells us that $K_{L_n}$ bounds a null-homologous smooth torus in $S^2\times D^2$.

\end{document}